\documentclass[12pt]{amsart}
\usepackage{amsmath, amsthm, amssymb}
\usepackage{amssymb}
\usepackage{amsfonts}
\usepackage{amscd}
\usepackage{mathrsfs}
\usepackage{latexsym}
\usepackage{amsmath}
\usepackage{txfonts}
\usepackage{tikz}
\usepackage{tikz-cd}
\usepackage{graphicx}
\usepackage{subfig}
\usepackage{xcolor}
\usepackage{pdfsync}
\usepackage{hyperref}
\usepackage[all, cmtip]{xy}

\hypersetup{
    linkbordercolor=red,          
    citebordercolor=blue        
}

\newtheorem{theorem}{Theorem}[section]
\newtheorem{lemma}[theorem]{Lemma}
\newtheorem{corollary}[theorem]{Corollary}
\newtheorem{proposition}[theorem]{Proposition}

\theoremstyle{definition}
\newtheorem{definition}[theorem]{Definition}

\theoremstyle{remark}
\newtheorem{remark}[theorem]{Remark}
\newtheorem{conjecture}[theorem]{Conjecture}

\newcommand{\sshf}[1]{\mathscr{O}_{#1}}
\newcommand{\shf}[1]{\mathscr{#1}}

\newcommand{\prj}[1]{\mathbb{P}^{#1}}

\newcommand{\iso}{\simeq}
\newcommand{\ses}[3]{0\rightarrow#1\rightarrow#2\rightarrow#3\rightarrow{0}}

\newcommand{\paren}[1]{\left(#1\right)}

\numberwithin{equation}{section}

\begin{document}
\allowdisplaybreaks
\title[projective normality of cyclic coverings]{On the projective normality of cyclic coverings over a rational surface}

\author{Lei Song}

\address{Department of Mathematics, University of California, Riverside, 900 University Ave., Riverside, CA 92521, USA}

\curraddr{School of Mathematics,
Sun Yat-sen University, No. 135 Xingang Xi Road, Guangzhou, Guangdong 510275, P.R. China}
\email{songlei3@mail.sysu.edu.cn}

\subjclass[2010]{Primary 14C20, 14J26, 14E20; Secondary 14N05}


\dedicatory{}

\keywords{projective normality, adjoint divisor, cyclic covering, rational surface}

\begin{abstract}
Let $S$ be a rational surface with $\dim|-K_S|\ge 1$ and let $\pi: X\rightarrow S$ be a ramified cyclic covering from a nonruled smooth surface $X$. We show that for any integer $k\ge 3$ and ample divisor $A$ on $S$, the adjoint divisor $K_X+k\pi^*A$ is very ample and normally generated. Similar result holds for minimal (possibly singular) coverings.
\end{abstract}

\maketitle

\section{Introduction}
 Given a complex projective variety $X$ and a base point free line bundle $L$ on $X$, $L$ is called \textit{normally generated} if the natural map
\begin{equation*}
   H^0(X, L)\otimes_{\mathbb{C}} H^0(X, L^{\otimes i})\rightarrow H^0(X, L^{\otimes {(i+1)}})
\end{equation*}
surjects for all $i>0$. A Cartier divisor $D$ is \textit{normally generated} if the associated line bundle $\sshf{X}(D)$ is base point free and normally generated. If $L$ is very ample and normally generated, then $X$ is embedded by the complete linear system $|L|$ as a projectively normal variety. The concept of normal generation was introduced by Mumford in \cite{Mumford69}.

A classical theorem of Castelnuovo, Mattuck, and Mumford says that on a smooth projective curve of genus $g$, any line bundle with degree at least $2g+1$ is normally generated. Also by the work of Reider \cite{Reider88}, it is well known that for a smooth projective surface $X$ and an ample divisor $A$, the adjoint divisor $K_X+rA$ is very ample provided $r\ge 4$. The considerations naturally lead to the conjecture, commonly attributed to S.~Mukai.

\begin{conjecture}[{cf.~\cite[Conj.~4.2]{EinLazarsfeld93}}]\label{Conj}
Let $X$ be a smooth projective surface and $A$ an ample divisor on $X$. Then for every integer $k\ge 4$, $K_X+kA$ is normally generated.
\end{conjecture}

The conjecture is true in many interesting cases by many people's work, but open in general, especially for surfaces of general type. A uniform approach to the problem for algebraic surfaces seems to be elusive by far.

The present paper is a continuation of \cite{RS16}. Blending the geometry of anticanonical rational surfaces and techniques from birational geometry, we extend the main result in \cite{RS16} to arbitrary ramified cyclic coverings over a rational surface $S$ with $\dim |-K_S|\ge 1$; in particular, we remove the minimality assumption for the coverings provided they are smooth.

The class of surfaces of general type which have covering structures over rational surfaces is ubiquitous in the geography and moduli of surfaces of general type. For instance, by \cite{Horikawa76}, for a minimal surface of general type satisfying $K^2_X=2p_g(X)-4$, the canonical system gives rise to a double covering structure of the surface or its canonical model (with mild singularities) over $\mathbb{P}^2$ or some rational ruled surface $\mathbb{F}_e$.

\begin{corollary}
Let $X$ be a nonruled smooth surface and $\pi: X\rightarrow S$ be a ramified cyclic covering over a rational surface $S$ with $\dim |-K_S|\ge 1$. Then for every $k\ge 3$ and ample divisor $A$ on $S$, $K_X+k\pi^*A$ is very ample and normally generated.
\end{corollary}

This immediately follows from the main theorem.

\begin{theorem}\label{main result}
Let $S$ be a rational surface with $\dim|-K_S|\ge 1$. Let $\pi: X\rightarrow S$ be a ramified $r$-cyclic covering of $S$ with the branch locus $\Gamma\in |rB|$ for some divisor $B$. Suppose that one of the following conditions holds
\begin{enumerate}
  \item $X$ is smooth and $H^0(S, \sshf{S}(mK_S+lB))\neq 0$ for some integers $m, l\ge 1$ \footnote{Since $-K_S$ is effective in our setting, the condition that $H^0(\sshf{S}(mK_S+lB))\neq 0$ for some integers $l, m\ge 1$ is equivalent to the condition that $H^0(\sshf{S}(K_S+nB))\neq 0$ for some integer $n\ge 1$. But we will keep the preceding one in the presentation, since it follows more immediately from the nonruledness of $X$.}.
  \item $K_X$ is nef. \footnote{$X$ is possibly singular, but in any event Cohen-Macaulay and Gorenstein.}
\end{enumerate}
Then for any divisor $A$ on $S$ with the property that $A^2\ge 7$ and $A\cdot C\ge 3$ for any curve $C$ on $S$, the divisor $K_X+\pi^*A$ is very ample and normally generated.
\end{theorem}

Note that if $K_X$ is nef, then $K_S+(r-1)B$ is nef, and hence $H^0(\sshf{S}(K_S+(r-1)B))\neq 0$ by the Riemann-Roch theorem.

The idea of the proof is similar to that of \cite{RS16}, that is, one needs to find an auxiliary cycle on $S$ to deal with the potentially low positivity of $B$ in establishing the surjectivity of the multiplication maps
\begin{eqnarray*}
    H^0(S, \sshf{S}(K_S+(k+1)B+A))&\otimes& H^0(S, \sshf{S}(K_S+kB+A))\\
    &\rightarrow& H^0(S, \sshf{S}(2K_S+(2k+1)B+2A))
\end{eqnarray*}
for various integer $k$. An ingredient in the present paper is an application of Zariski decomposition to the divisor $B$, which enables us to define an auxiliary rational cycle on $S$ (see (\ref{auxiliary cycle}) and Theorem \ref{Artin thm} for notion of rational cycle) that works in the argument for all situations except for the case when $|-K_S|$ induces an elliptic fibration of $S$ over $\prj{1}$. This extremal case is treated separately in the end of Section 4. It may be worth pointing out that the auxiliary cycle is not reduced in general; and an analogous situation as in the theorem where the base surface $S$ is K3 or abelian is arguably less involved.

In the situation of double coverings, which was studied in \cite{RS16}, Harbourne's vanishing theorem \cite[Thm.~3.1]{Harbourne97} for nef divisors on anticanonical rational surfaces is extensively used. By contrast, for arbitrary cyclic coverings, $\mathbb{Q}$-divisors appear naturally, and Harbourne's vanishing unfortunately does not extend to the perturbations of nef divisors by fractional divisors. For this reason, we instead rely more on the Kawamata--Viehweg vanishing, this makes the choice of auxiliary cycle trickier, though.

{\it Acknowledgments.}
I am grateful to Shin-Yao Jow, Bangere Purnaprajna, and Yuan Wang for very helpful discussions and communications. Special thanks goes to Brian Harbourne, whose illuminating examples clarify my questions on anticanonical rational surfaces. Last but not least, I would like to thank the referee for the suggestions, which improve the exposition of the article.

\section*{Conventions and notation}
We work over the field of complex numbers $\mathbb{C}$. All surfaces are projective and smooth unless otherwise specified. A curve or cycle on a surface means a nonzero effective divisor on the surface. Given a sheaf $\shf{F}$ of $\sshf{X}$-modules, the notation $h^i(X, \shf{F})$ denotes the dimension of the cohomology group $H^i(X, \shf{F})$. When the context is clear, we simply write $H^i(\shf{F})$ for $H^i(X, \shf{F})$ and similarly for $h^i(\shf{F})$. We say a divisor class $D$ is \textit{effective} if $h^0(\sshf{S}(D))>0$. We do not distinguish between a divisor and its associated line bundle. For $\mathbb{Q}$-divisors, we will also interchangeably use the symbol $\sim$ for linear equivalence and $\equiv$ for numerical equivalence, as the two equivalences coincide on regular surfaces. For a $\mathbb{Q}$-divisor $D$, $\left\lfloor D\right\rfloor$ ($\left\lceil D\right\rceil$, resp.) denotes the round-down (round-up, resp.) of $D$, and $\{D\}$ denotes the fractional part of $D$.

\section{Preliminaries}
Given a surface $S$, let $K_S$ denote a canonical divisor of $S$. A rational surface $S$ is called \textit{anticanonical} if its anticanonical system $|-K_S|$ is nonempty. Examples of anticanonical rational surfaces include del Pezzo surfaces, blowing-ups of $\mathbb{F}_e$ along at most 8 points, and proper toric surfaces.

We collect some facts on anticanonical rational surfaces which will be needed in the subsequent sections.
\subsection{Adjoint divisors on anticanonical rational surfaces}
First we recall two properties of adjoint divisors.

\begin{proposition}[{\cite[Prop.~3.7]{RS16}}]\label{inequality1}
Let $S$ be an anticanonical rational surface and $A$ be a divisor on $S$ with the property that $A\cdot C\ge 3$ for any curve $C$ on $S$. Then $(K_S+A)\cdot(-K_S)\ge 3$ unless
\begin{enumerate}
  \item $S=\prj{2}$, $A=\sshf{\prj{2}}(3)$,
  \item $K^2_S=1, A=-3K_S$.
\end{enumerate}
In case $(2)$, $-K_S$ is ample and has a unique base point.
\end{proposition}

\begin{proposition}[{\cite[Prop.~3.8]{RS16}}]\label{base point freeness of K+A}
Let $S$ be an anticanonical rational surface and $A$ be a divisor on $S$ with the property that $A\cdot C\ge 3$ for any curve $C$ on $S$. Then the following are equivalent:
\begin{enumerate}
  \item $K_S+A$ is nef;
  \item $K_S+A$ is base point free;
  \item $A^2\ge 7$.
\end{enumerate}
 Moreover if any of the equivalent conditions holds, then $K_S+A$ is ample unless $(S, A)=(\prj{2}, \sshf{\prj{2}}(3))$.
\end{proposition}

\subsection{Fixed part of the anticanonical linear system}

Anticanonical rational surfaces have many common behaviors with K3 surfaces. But for anticanonical rational surfaces, $|-K_S|$ can have a fixed part. Put
\begin{equation*}
    |-K_S|=F+|M|,
\end{equation*}
where $F$ is the fixed part and $M$ the moving part (hence is nef). What is worse, $F$ is neither reduced nor irreducible in general, see \cite[Ex.~3.4]{RS16} for an example due to B. Harbourne.

In case that $\dim|-K_S|\ge 1$ and $F\neq\emptyset$, $F$ satisfies the property that $H^1(\sshf{F})=0$, as the following criterion implies.

\begin{lemma}\label{criterion for rational cycle}
Let $S$ be a surface with $q=p_g=0$. Let $C$ be a curve on $S$. Then $H^1(\sshf{C})=0$ if and only if $H^0(\sshf{S}(K_S+C))=0$.
\end{lemma}
\begin{proof}
The short exact sequence
\begin{equation*}
    \ses{\sshf{S}(-C)}{\sshf{S}}{\sshf{C}},
\end{equation*}
yields the exact sequence
\begin{equation*}
     H^1(\sshf{S})\rightarrow H^1(\sshf{C})\rightarrow  H^2(\sshf{S}(-C))\rightarrow H^2(\sshf{S})\rightarrow 0.
\end{equation*}
Since $H^1(\sshf{S})=H^2(\sshf{S})=0$, one has
\begin{equation*}
    H^1(\sshf{C})\iso H^2(\sshf{S}(-C))\iso H^0(\sshf{S}(K_S+C))^{*},
\end{equation*}
by Serre duality.
\end{proof}

Curves on a surface with the property $H^1(\sshf{C})=0$ was studied by M.~Artin.
\begin{theorem}[{\cite[Thm.~1.7]{Artin62}}]\label{Artin thm}
Given a positive cycle $Z=\sum^s_{i=1} a_i Z_i$ on a surface $S$ which is smooth at $\text{Supp}(Z)$. The following are equivalent:
\begin{enumerate}
  \item $H^1(\sshf{Z})=0$.
  \item The group homomorphism $d: H^1(\sshf{Z}^*)\rightarrow \mathbb{Z}^s$ by
  \begin{equation*}
    \shf{L}\mapsto (\deg{(\shf{L}|_{Z_1})}, \deg{(\shf{L}|_{Z_2})}, \cdots, \deg{(\shf{L}|_{Z_s})})
  \end{equation*}
  is an isomorphism.
  \item For every positive cycle $Y\le Z$, the arithmetic genus $p_a(Y)\le 0$.
\end{enumerate}
\end{theorem}

Note that the theorem implies that if $H^1(\sshf{Z})=0$, then for every positive cycle $Y\le Z$, $H^1(\sshf{Y})=0$. In particular, each irreducible component of $Z$ is a smooth rational curve. This is also obvious in the special case when $q=p_g=0$ in view of Lemma \ref{criterion for rational cycle}.
Following \cite{Artin62}, we define
\begin{definition}
Let $C$ be a curve on a surface $S$. $C$ is called a $\textit{rational cycle}$ if $H^1(C, \sshf{C})=0$.
\end{definition}

The following simple observation is useful
\begin{lemma}\label{normal crossing}
Let $Z$ be a rational cycle. Then for any two distinct components $C, C'$ of $Z$, the intersection number $C\cdot C'\le 1$.
\end{lemma}
\begin{proof}
Since $C+C'$ is a sub-cycle of $Z$, by Theorem \ref{Artin thm}, $p_a(C+C')\le 0$, and hence $(C+C'+K_S)\cdot(C+C')\le -2$. The assertion follows using the fact that $C$ and $C'$ are smooth rational curves.
\end{proof}

\begin{lemma}\label{M.F}
Suppose $\dim|-K_S|\ge 1$ and $F\neq \emptyset$. Then
\begin{equation*}
    M\cdot F= 2h^0(\sshf{F}) \quad \text{and }\quad  h^0(\sshf{M})=h^0(\sshf{F}).
\end{equation*}
\end{lemma}
\begin{proof}
Since $\dim|-K_S|\ge 1$, one has that $M\neq 0$. The exact sequence
\begin{equation*}
    \ses{\sshf{S}(-F)}{\sshf{S}}{\sshf{F}}
\end{equation*}
yields
\begin{eqnarray*}
   h^0(\sshf{F})&=& 1+ h^1(\sshf{S}(-F))\\
   &=& 1-\chi(\sshf{S}(-F))\quad\quad\: \text{note } h^0(\sshf{S}(-F))=h^2(\sshf{S}(-F))=0\\
   &=& -\frac{(-F)\cdot(-F-K_S)}{2}\\
   &=&\frac{M\cdot F}{2}.
\end{eqnarray*}
This gives the first assertion that $M\cdot F=2h^0(\sshf{F})$. For the second, consider the exact sequence
\begin{equation*}
    \ses{\sshf{S}(-M)}{\sshf{S}}{\sshf{M}}.
\end{equation*}
Using the above result, we get that
\begin{equation*}
    h^0(\sshf{M})=1+h^1(\sshf{S}(-M))=1+h^1(\sshf{S}(-F))=h^0(\sshf{F}),
\end{equation*}
where the second identity is by Serre duality. This finishes the proof.
\end{proof}

Since $M\cdot (-K_S)\ge M\cdot F\ge 2$ by Lemma \ref{M.F}, we deduce that $M$ is base point free and all higher cohomology vanishes by \cite[Thm 3.1]{Harbourne97}. We put $s:=h^0(\sshf{F})=h^0(\sshf{M})$. A general member in $|M|$ can be written as $\sum^s_{i=1} M_i$, where $M_i$ are smooth irreducible curve. Since $M^2_i=M_i\cdot M\ge 0$ for all $i$, by Hodge index theorem, there are two possibilities:
\begin{enumerate}
  \item  $s>1$, $M^2_i=0$ for all $i$.
  \item  $s=1$, $M_1^2\ge 0$.
\end{enumerate}
By the same argument as above, we can get that
\begin{equation*}
    \frac{F\cdot M_i}{2}=h^0(\sshf{F+\sum_{j\neq i} M_j})\ge 1.
\end{equation*}
But by Lemma \ref{M.F}, $F\cdot M=2s$, so $F\cdot M_i=2$, for each $i$. In either case, $M_i\iso \mathbb{P}^1$.
\subsection{Multiplication maps of line bundles}

For regular varieties, sometimes one can reduce the surjectivity of considered multiplication maps of line bundles to the surjectivity of multiplication maps on a hypersurface, as the following lemma shows.
\begin{lemma}[{\cite[p.~154]{GalPurna99}}]\label{GP Lemma}
Let $X$ be a smooth variety with $H^1(\sshf{X})=0$. Let $E$ be a vector bundle and $L\iso\sshf{X}(D)$ be a base point free line bundle with the property that $H^1(E\otimes L^{-1})=0$. If the natural map $H^0(E|_D)\otimes H^0(L|_D)\rightarrow H^0(E\otimes L|_D)$ is surjective, then so is the natural map $H^0(E)\otimes H^0(L)\rightarrow H^0(E\otimes L)$.
\end{lemma}

We give a criterion for the surjectivity of multiplication maps on anticanonical rational surfaces in the following lemma, which strengthens \cite[Lemma 5.5]{RS16}. Here, we remove the assumption that $B'$ is effective, making the criterion more flexible in applications. For the reader's convenience, we give a sketch of its proof, while details can be found in \cite[pp. 262, 263]{RS16}.
\begin{lemma}\label{surjectiveness of multiplication map2}
Let $S$ be an anticanonical rational surface and $L$ be a divisor on $S$ such that $|K_S+L|$ contains a nonsingular curve. Let $B'$ be a divisor on $S$ with the property that $H^1(\sshf{S}(B'))=0$. Suppose $(K_S+L)\cdot (-2K_S+B')\ge 5$ and $(K_S+L)\cdot (-K_S+B')\ge 2$. Then the natural map
 \begin{equation*}
    H^0(\sshf{S}(K_S+L+B'))\otimes H^0(\sshf{S}(K_S+L))\rightarrow H^0(\sshf{S}(2K_S+2L+B'))
 \end{equation*}
 is surjective.
\end{lemma}

\begin{proof}[Sketch of the Proof]
Since $H^1(B')=0$, by Lemma \ref{GP Lemma}, the problem can be reduced to the surjectivity of
\begin{equation*}
    H^0((K_S+L+B')|_{\Delta})\otimes  H^0((K_S+L)|_{\Delta})\rightarrow  H^0((2K_S+2L+B')|_{\Delta}),
\end{equation*}
where $\Delta\in|K_S+L|$ is a nonsingular curve.

To apply Green's $H^0$-lemma \cite{Green84}, we shall establish the inequality
\begin{equation}\label{inequality to apply H^0}
    h^1(B'|_{\Delta})\le h^0((K_S+L)|_{\Delta})-2.
\end{equation}
By Riemann-Roch, $h^0((K_S+L)|_{\Delta})\ge (K_S+L)^2+1-g(\Delta)$, where $g(\Delta)$ is the genus of $\Delta$. On the other hand, $h^1(B'|_{\Delta})=h^0((\Delta+K_S-B')|_{\Delta})$ by duality. For an upper bound of $h^0((\Delta+K_S-B')|_{\Delta})$, we discuss by cases.

(i) If $\deg (\Delta+K_S-B')|_{\Delta}\le \deg K_{\Delta}$, this is treated in \cite{RS16} and the assumption that $(K_S+L)\cdot (-2K_S+B')\ge 5$ is utilized.

(ii) If $\deg (\Delta+K_S-B')|_{\Delta}>\deg K_{\Delta}$, then
  \begin{equation*}
    h^0((\Delta+K_S-B')|_{\Delta})=\chi((\Delta+K_S-B')|_{\Delta})=(2K_S+L-B')\cdot (K_S+L)+1-g(\Delta).
  \end{equation*}
Therefore
\begin{equation*}
    (h^0((K_S+L)|_{\Delta})-2)-h^1(B'|_{\Delta})\ge (K_S+L)\cdot (-K_S+B')\ge 0
\end{equation*}
by the assumption. So (\ref{inequality to apply H^0}) holds, and this completes the proof.
\end{proof}

For any rational cycle $C$ on a surface, we have the following criterion for surjectivity of multiplication maps over $C$.
\begin{lemma}\label{surjectiveness of multiplication map}
Let $S$ be a surface and $C$ a curve with $h^1(\sshf{C})=0$. Put $C=\sum a_i\Gamma_i$, where $\Gamma_i$ are irreducible components of $C$. Let $L_1, L_2$ be two line bundles on $S$ such that $L_1\cdot\Gamma_i>0, L_2\cdot \Gamma_i\ge 0$ for every $i$. Suppose $H^1(L_2(-\Gamma_i))=0$ for every $i$. Then the natural map
\begin{equation*}
    H^0({L_1}|_C)\otimes H^0(L_2)\rightarrow H^0((L_1+L_2)|_C)
\end{equation*}
is surjective.
\end{lemma}
\begin{proof}
The assumption on $L_2$ is slightly different from the one in {\cite[Lemma 5.2]{RS16}}, but the proof is exactly the same. Note that if $L_2$ is ample and base point free and if $-K_S-\Gamma_i$ is effective, then $H^1(L_2(-\Gamma_i))=0$ by \cite[Lemma 5.1]{RS16}.
\end{proof}

\section{Construction of rational cycles, properties of adjoint divisors}
We start by fixing notations for the remaining sections. Let $S$ be an anticanonical rational surface with $\dim|-K_S|\ge 1$. Let $r\ge 2$ be a fixed integer and $\pi: X\rightarrow S$ be a ramified $r$-cyclic covering (possibly singular) over $S$ with the branch locus $\Gamma\in|\sshf{S}(rB)|$ for some integral divisor $B$. Then $K_X\sim\pi^*(K_S+(r-1)B)$ and $\pi_*\sshf{X}\iso \oplus^{r-1}_{i=0}\sshf{S}(-iB)$ as $\sshf{S}$-modules. Moreover let $A$ be an ample divisor on $S$ satisfying that $A^2\ge 7$ and $A\cdot C\ge 3$ for any curve $C$ on $S$.

Since $B$ is a pseudo-effective integral divisor, we have the Zariski decomposition of $B$ as a summation of $\mathbb{Q}$-divisors
\begin{equation*}
    B=P+N,
\end{equation*}
with the property that
\begin{itemize}
  \item[(i)] $P$ is nef;
  \item[(ii)] $N=\sum^{s}_{i=1}a_iE_i$ is effective. If $N\neq 0$, then the intersection matrix $\|(E_i\cdot E_j)\|$ is negative definite;
  \item[(iii)] For each $1\le i\le s$, the intersection number $P\cdot E_i=0$.
\end{itemize}

\begin{lemma}\label{nontriviality of P}
Assume $\dim |-K_S|\ge 1$. Suppose $H^0(\sshf{S}(mK_S+lB))\neq 0$ for some integers $l, m\ge 1$. Then $P\neq 0$.
\end{lemma}
\begin{proof}
Suppose to the contrary that $P=0$, then $N=B$ is integral. By \cite[Prop. 2.3.21]{Lazarsfeld04}, the natural map
\begin{equation*}
   H^0(S, \sshf{S})\iso H^0(S, \sshf{S}(pB-\left\lceil pN\right\rceil))\rightarrow H^0(S, \sshf{S}(pB))
\end{equation*}
is bijective for any $p\ge 1$. Since $mK_S+lB$ is effective, we have
\begin{equation*}
   h^0(\sshf{S}(lB))=h^0(\sshf{S}((mK_S+lB)+m(-K_S)))\ge h^0(\sshf{S}(-K_S))\ge 2.
\end{equation*}
This gives a contradiction. Thus $P\neq 0$.
\end{proof}

We define a cycle depending on both $|-K_S|$ and $B$
\begin{equation}\label{auxiliary cycle}
 Z=F+\left\lfloor N\right\rfloor.
\end{equation}
Here, we allow $F=\emptyset$, that is, when $|-K_S|$ has no fixed part.

\begin{proposition}\label{vanishing of H^1(B')}
Suppose that $|-K_S|$ has the fixed part $F$ and $\dim |-K_S|\ge 1$. Assume that $B$ satisfies $H^0(mK_S+lB)\neq 0$ for some integers $l, m\ge 1$. Then
\begin{equation*}
    H^1(\sshf{S}(B-Z))=0.
\end{equation*}
\end{proposition}
\begin{proof}
We have the identity for $\mathbb{Q}$-divisors
\begin{equation*}
  B-(F+\left\lfloor N\right\rfloor)= K_S+(P+M)+\{N\}.
\end{equation*}
We claim that $M\cdot P>0$. Suppose to the contrary that $M\cdot P=0$, then since $P$ and $M$ are nef and thanks to Lemma \ref{nontriviality of P}, they are not numerically trivial, and hence the Hodge index theorem implies that $M=cP$ for some $c\in\mathbb{Q}_{+}$ and $M^2=0$, from which it follows that
\begin{equation*}
    (mK_S+lB)\cdot M=mK_S\cdot M=-mF\cdot M<0,
\end{equation*}
where the inequality is by Lemma \ref{M.F}. This contradicts the assumption. Therefore $(P+M)^2>0$, and hence $P+M$ is a nef and big $\mathbb{Q}$-divisor. Then Kawamata--Viehweg vanishing\footnote{On surfaces, simple normal crossing condition for the fractional part $\{N\}$ is not needed.} (cf. for instance \cite[9.1.18]{Lazarsfeld04II}) applies.
\end{proof}

\begin{proposition}\label{rational cycle}
Let $Z$ be a cycle as defined in $(\ref{auxiliary cycle})$. Then
\begin{equation*}
    H^1(Z, \sshf{Z})=0.
\end{equation*}
\end{proposition}
\begin{proof}
We shall show that $H^0(\sshf{S}(K_S+Z))$=0, then the assertion is a consequence of Lemma \ref{criterion for rational cycle}. There are three cases:

(i) $|-K_S|$ has the fixed part $F$.
Since
\begin{equation*}
    (K_S+F+\left\lfloor N\right\rfloor)\cdot P=-M\cdot P<0,
\end{equation*}
as the proof of Proposition \ref{vanishing of H^1(B')} has shown, thus we have $H^0(\sshf{S}(K_S+F+\left\lfloor N\right\rfloor))=0$.

(ii) $|-K_S|$ has no fixed part and $K^2_S>0$. Then $-K_S$ is big and nef. Since
\begin{equation*}
    (K_S+\left\lfloor N\right\rfloor)\cdot P=K_S\cdot P<0,
\end{equation*}
by the Hodge index theorem and nefness of $P$, we deduce that $H^0(\sshf{S}(K_S+\left\lfloor N\right\rfloor))=0$.

(iii) $|-K_S|$ has no fixed part and $K^2_S=0$. If $P^2>0$, then we can proceed as (ii). Thus we can assume in addition that
\begin{equation*}
    (-K_S)\cdot P=P^2=0.
\end{equation*}
Then by the Hodge index theorem we deduce from that $P\equiv c(-K_S)$ for some rational number $c>0$ (Recall $P\nequiv 0$ by Lemma \ref{nontriviality of P}).
Now suppose to the contrary that $K_S+\left\lfloor N\right\rfloor\sim D$ for some effective divisor $D$, so $\left\lfloor N\right\rfloor\neq 0$. Since
\begin{equation*}
   D\cdot \left\lfloor N\right\rfloor= \paren{K_S+\left\lfloor N\right\rfloor}\cdot \left\lfloor N\right\rfloor=\left\lfloor N\right\rfloor^2<0,
\end{equation*}
it follows that $D$ and $\left\lfloor N\right\rfloor$ must have common components. Denote by $D'$ (resp. $\left\lfloor N\right\rfloor'$) the divisor obtained from $D$ (resp. $\left\lfloor N\right\rfloor$) by subtracting the greatest common subdivisor of $D$ and $\left\lfloor N\right\rfloor$. We have
\begin{equation*}
    K_S+\left\lfloor N\right\rfloor'\sim D'.
\end{equation*}
Again $\left\lfloor N\right\rfloor'\neq 0$. So
\begin{equation*}
    D'\cdot \left\lfloor N\right\rfloor'= \paren{K_S+\left\lfloor N\right\rfloor'}\cdot \left\lfloor N\right\rfloor'=\paren{\left\lfloor N\right\rfloor'}^2<0,
\end{equation*}
which is impossible because $D'$ and $\left\lfloor N\right\rfloor'$ have no common component. Therefore, $H^0(\sshf{S}(K_S+\left\lfloor N\right\rfloor))=0$ and the proof is complete.
\end{proof}

\begin{remark}\label{rmk on rational cycle}
From the proof, it is clear that for any positive integer $k$, it holds that $H^1(\sshf{F+{\left\lfloor kN\right\rfloor}})=0$.
\end{remark}

Next we study the positivity of $K_S+kB+A$, which will be needed in the next section.

\begin{proposition}\label{normal generation of K+kB+A}
Suppose that one of the following conditions holds:
\begin{enumerate}
  \item $X$ is smooth.
  \item $K_S+(r-1)B$ is nef.
\end{enumerate}
Then for any $0< k\le r-1$,
\begin{itemize}
  \item[(i)] the divisor $K_S+kB+A$ is very ample and normally generated.
  \item[(ii)] $(K_S+kB+A)^2>1$, unless $(S, A, B, k)=(\prj{2}, \sshf{\prj{2}}(3), \sshf{\prj{2}}(1), 1)$.
\end{itemize}
\end{proposition}

\begin{proof}
(ii) is evident once (i) is proved.
If $S=\prj{2}$, the assertion (i) is trivial; so we assume $S\neq \prj{2}$.

(1) Suppose that $X$ is smooth. By the assumption on $A$ and \cite[Thm.~1]{Reider88}, we deduce that $K_X+\pi^*A$ is very ample, so $K_S+(r-1)B+A$ is ample, as $\pi$ is finite. In view of Proposition \ref{base point freeness of K+A}, for any $0< k\le r-1$,
\begin{equation}\label{decomposition of adjoint divisor}
    K_S+kB+A=\frac{r-1-k}{r-1}(K_S+A)+\frac{k}{r-1}(K_S+(r-1)B+A)
\end{equation}
is ample. Once we show that $(K_S+kB+A)\cdot (-K_S)\ge 3$, then the assertion follows from \cite[Thm.~III.1]{Harbourne97} and \cite[Thm.~1.3, Prop.~1.22]{GalPurna01}.

To this end, we first prove $(K_S+(r-1)B+A)\cdot (-K_S)\ge 3$ as follows. Recall by Proposition \ref{inequality1}, $(K_S+A)\cdot (-K_S)\ge 2$. If $(K_S+A)\cdot (-K_S)=2$, then it hold that $K^2_S=1, A=-3K_S$. Therefore
\begin{equation*}
    (K_S+(r-1)B+A)\cdot(-K_S)=2+(r-1)B\cdot(-K_S)\ge 3.
\end{equation*}
If $(K_S+A)\cdot (-K_S)\ge 3$ and $|-K_S|$ has no fixed part, then
\begin{equation*}
    (K_S+(r-1)B+A)\cdot(-K_S)\ge (K_S+A)\cdot (-K_S)\ge 3.
\end{equation*}
If $(K_S+A)\cdot (-K_S)\ge 3$ and $|-K_S|$ has a fixed part $F$, then
\begin{equation*}
   (K_S+(r-1)B+A)\cdot(-K_S)=(K_S+(r-1)B+A)\cdot (M+F)\ge 2.
\end{equation*}
Suppose that the intersection number was 2. It is then evident that $(K_S+(r-1)B+A)\cdot M=1$, $F$ is irreducible and $F\cdot B=0$, which implies that $F\subseteq \Gamma\in |rB|$. Hence $rB\cdot M\ge F\cdot M\ge 2$ (see Lemma \ref{M.F}). In particular, $B\cdot M>0$. It follows that $(K_S+(r-1)B+A)\cdot M=(K_S+A)\cdot M+(r-1)B\cdot M\ge 2$, a contradiction.

For general $k$, we have
\begin{eqnarray*}
    &&(K_S+kB+A)\cdot (-K_S)\\
    &=& \frac{r-1-k}{r-1}(K_S+A)\cdot (-K_S)+\frac{k}{r-1}(K_S+(r-1)B+A)\cdot (-K_S)\\
    &\ge& 2\paren{\frac{r-1-k}{r-1}}+3\paren{\frac{k}{r-1}}\\
    &>&2.
\end{eqnarray*}

(2) Next suppose that $K_S+(r-1)B$ is nef. Since
\begin{equation*}
    K_S+kB+A=\frac{k}{r-1}(K_S+(r-1)B)+\frac{r-1-k}{r-1}(K_S+A)+\frac{k}{r-1}A,
\end{equation*}
we get that $K_S+kB+A$ is ample. It follows that
\begin{eqnarray*}p
   &&(K_S+kB+A)\cdot(-K_S)\\
   &\ge & \frac{r-1-k}{r-1}(K_S+A)\cdot(-K_S)+ \frac{k}{r-1}A\cdot(-K_S) \\
   &\ge &  2\paren{\frac{r-1-k}{r-1}}+ 3\paren{\frac{k}{r-1}}\quad\quad\quad\quad\text{by Proposition \ref{inequality1}}\\
   &>& 2.
\end{eqnarray*}
Again by \cite[Thm.~III.1]{Harbourne97} and \cite[Thm.~1.3, Prop.~1.22]{GalPurna01}, $K_S+kB+A$ is very ample and normally generated. This completes the proof.
\end{proof}

The technical lemma on certain intersection numbers will be needed in the next section for applying Lemma \ref{surjectiveness of multiplication map2}.

\begin{lemma}\label{intersection inequality}
Let $L$ be an ample divisor on $S$. Then $L\cdot (-K_S+B-Z)\ge 2$. Moreover, if $|-K_S|$ has the fixed part and $L^2>1$, then
$ L\cdot (-2K_S+B-Z)\ge 5$.
\end{lemma}

\begin{proof}
Since
\begin{equation*}
    L \cdot(-K_S+B-Z)=L\cdot(P+M+\{N\})\ge  L\cdot(P+M)>1,
\end{equation*}
we conclude that $L\cdot(-K_S+B-Z)\ge 2$.

For the second inequality, suppose to the contrary that
$L\cdot (-2K_S+B-Z)\le 4$. Then $L\cdot F=L\cdot M=1$, implying that $F$ is integral and so is every member of $|M|$. We claim that $M^2>0$. For otherwise, $|M|: S\rightarrow \mathbb{P}^1$ is a rational ruled surface with invariant $e\ge 3$. Then computation shows that $M^2=e+4>0$, a contradiction. Then, Hodge index theorem forces that $L^2=M^2=1$, contradicting the assumption that $L^2>1$. Therefore $L\cdot (-2K_S+B-Z)\ge 5$.
\end{proof}

The following result will be needed in the next section for applying Lemma \ref{surjectiveness of multiplication map}.
\begin{proposition}\label{H^1 vanishing}
Keep the assumption as in Proposition $\ref{normal generation of K+kB+A}$. Then for any irreducible component $E$ of $Z$ and any $0\le k\le r-1$, we have the vanishing
\begin{equation*}
    H^1(S, \sshf{S}(K_S+kB+A-E))=0.
\end{equation*}
\end{proposition}
\begin{proof}
Again, we can assume that $S\neq \mathbb{P}^2$. In view of Proposition \ref{base point freeness of K+A} and \ref{normal generation of K+kB+A}, $\sshf{S}(K_S+kB+A)$ is ample and base point free; so according to \cite[Lemma 5.1]{RS16}, the vanishing holds so long as $-K_S-E$ is effective. Therefore, we henceforth suppose $E\subseteq\text{Supp}{\left\lfloor N\right\rfloor}$.

Put $\shf{L}=\sshf{S}(K_S+kB+A-E)$. We will take a rational cycle $C$ such that $H^1(\shf{L}(-C))=0$, and $\shf{L}$ intersects each component of $C$ nonnegatively. Hence $H^1(\shf{L}|_{C})=0$ by \cite[Lemma 7]{Harbourne96}. Then the vanishing of $H^1(\shf{L})$ will follow from the exact sequence
\begin{equation*}
    \ses{\shf{L}(-C)}{\shf{L}}{\shf{L}|_C}.
\end{equation*}

To this end, it suffices to take $C=\left\lfloor kN\right\rfloor-E$, which is rational (cf.~Remark \ref{rmk on rational cycle}). Since
\begin{equation*}
    K_S+kB+A-E-C=K_S+(kP+A)+\{kN\},
\end{equation*}
by Kawamata--Viehweg vanishing, $H^1(\shf{L}(-C))=0$ holds. On the other hand, for each component $\Gamma$ of $\left\lfloor kN\right\rfloor$, if $\Gamma\neq E$, then $E\cdot \Gamma\le 1$ by Lemma \ref{normal crossing}. Otherwise $\Gamma\cdot E=E^2<0$. Thus, in any event,
\begin{equation*}
    (K_S+kB+A-E)\cdot \Gamma\ge 0.
\end{equation*}
This completes the proof.
\end{proof}


\section{Proof of Theorem \ref{main result}}

Keep the notations as in Section 3.
\begin{proof}
Note to begin with that $K_X+\pi^*A=\pi^*(K_S+(r-1)B+A)$ is ample and base point free, thanks to Proposition \ref{normal generation of K+kB+A} (i)
and the finiteness of $\pi$. Therefore, the very ampleness of $K_X+\pi^*A$ will be a consequence of normal generation by \cite[p.~38]{Mumford69} (In case $X$ is smooth, very ampleness alternatively follows from Reider's criterion \cite{Reider88}).

To prove the normal generation, we shall show that the natural map
\begin{equation}\label{surjectivity}
    H^0(\sshf{X}\paren{k(K_X+\pi^*A)})\otimes H^0(\sshf{X}\paren{K_X+\pi^*A})\rightarrow  H^0(\sshf{X}\paren{(k+1)(K_X+\pi^*A)})
\end{equation}
surjects for all $k\ge 1$.

For $k\ge 3$, the surjectivity of (\ref{surjectivity}) simply follows from  Castelnuovo--Mumford regularity property. In fact, for $i=1, 2$, we have
\begin{eqnarray*}
  && H^i(\sshf{X}\paren{(k-i)(K_X+\pi^*A)})\\
  &\iso& \bigoplus^{r-1}_{j=0}  H^i(\sshf{S}\paren{(k-i)(K_S+(r-1)B+A)-jB})\\
  &=&0
\end{eqnarray*}
by Proposition \ref{normal generation of K+kB+A} and \cite[Thm~III.1]{Harbourne97}. So it remains to treat the cases $k=1, 2$.  We only treat the case $k=1$, since the case $k=2$ is analogous and easier.

Pushing down (\ref{surjectivity}) to $S$, it is sufficient to show the surjectivity of
\begin{equation*}
    H^0(\sshf{S}(K_S+kB+A))\otimes  H^0(\sshf{S}(K_S+kB+A))\xrightarrow{f_1}  H^0(\sshf{S}(2K_S+2kB+2A))
\end{equation*}
for all $\left\lceil\frac{r-1}{2}\right\rceil\le k\le r-1$,
and that of
\begin{equation*}
    H^0(\sshf{S}(K_S+(k+1)B+A))\otimes H^0(\sshf{S}(K_S+kB+A))\xrightarrow{f_2} H^0(\sshf{S}(2K_S+(2k+1)B+2A)),
\end{equation*}
for all $\left\lceil\frac{r-2}{2}\right\rceil\le k\le r-2$.

According to Proposition \ref{normal generation of K+kB+A}, $f_1$ is surjective. For $f_2$, we will take a curve $C$ on $S$ (to be specified later) that makes rows in the commutative diagram exact
\begin{equation}\label{commutative diagram}
\includegraphics[trim=50mm 213mm 10mm 43mm, clip]{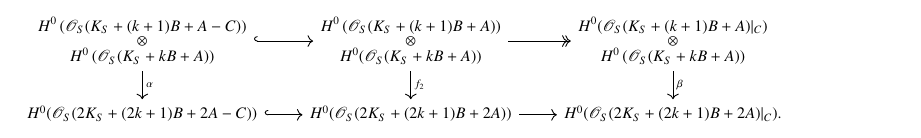}
\end{equation}
Then by the snake lemma, the surjectivity of $f_2$ will follow, once we establish the surjectivity of the column maps $\alpha$ and $\beta$.

If $S=\mathbb{P}^2$, then $f_2$ is clearly surjective. We assume henceforth that $S\neq \mathbb{P}^2$. There are three cases:
\vspace{0.2cm}

Case I: $|-K_S|$ has the fixed part.

\vspace{0.2cm}

Put $|-K_S|=F+|M|$. In this case, we take $C=Z$, as defined in (\ref{auxiliary cycle}), and put $B'=B-C$. Since
\begin{equation*}
    K_S+kB+A+B'\sim K_S+(K_S+kB+A)+(P+M)+\{N\}
\end{equation*}
and $K_S+kB+A$ is ample by Proposition \ref{normal generation of K+kB+A}(i),
Kawamata--Viehweg vanishing implies that $H^1(\sshf{S}(K_S+kB+A+B'))=0$, and hence the top row of diagram (\ref{commutative diagram}) is exact.

To show the surjectivity of $\alpha$, we will apply Lemma \ref{surjectiveness of multiplication map2}. It has been shown that $H^1(\sshf{S}(B'))=0$ in Proposition \ref{vanishing of H^1(B')}. If $k>0$, then by Lemma \ref{intersection inequality} and Proposition \ref{normal generation of K+kB+A}(ii), we have $(K_S+kB+A)\cdot(-K_S+B')\ge 2$ and $(K_S+kB+A)\cdot (-2K_S+B')\ge 5$. When $k=0$, verification for the second intersection inequality is slightly different. Since
\begin{equation*}
    (K_S+A)\cdot(-2K_S+B')\ge(K_S+A)\cdot (F+2M+P),
\end{equation*}
it suffices to show that $(K_S+A)\cdot F\ge 2$. To this end, we can assume $F$ is integral. But then as a fixed smooth rational curve on $S$, $F^2<0$. Using the adjunction formula and the assumption on $A$, we deduce that
\begin{equation*}
    (K_S+A)\cdot F=-2-F^2+A\cdot F\ge 1-F^2\ge 2
\end{equation*}
as desired.

For the surjectivity of $\beta$, we will apply Lemma \ref{surjectiveness of multiplication map}. Since $H^1(\sshf{C})=0$ by Proposition \ref{rational cycle} and since for each component $\Gamma_i$ of $C$, $H^1(\sshf{S}(K_S+kB+A-\Gamma_i))=0$ holds by Proposition \ref{H^1 vanishing}, we can apply Lemma \ref{surjectiveness of multiplication map} by setting $L_1=\sshf{S}(K_S+(k+1)B+A)$ and $L_2=\sshf{S}(K_S+kB+A)$.
\vspace{0.2cm}

Case II: $|-K_S|$ has no fixed part and $K^2_S>0$.

\vspace{0.2cm}

In this case, $-K_S$ is big and nef. We take $C=\left\lfloor N\right\rfloor$, which is a rational cycle by Proposition \ref{rational cycle}. Since
\begin{equation*}
    B'=B-C\sim K_S+(-K_S+P)+\{N\},
\end{equation*}
it follows that $H^1(\sshf{S}(B'))=0$ by Kawamata--Viehweg vanishing. By Proposition \ref{inequality1}, we see that
\begin{equation*}
    (K_S+kB+A)\cdot(-K_S)\ge (K_S+A)\cdot (-K_S)\ge 2;
\end{equation*}
and since $B'$ is pseudo-effective, it holds that $(K_S+kB+A)\cdot B'\ge 1$.
So the conditions for intersection numbers in Lemma \ref{surjectiveness of multiplication map2} hold in this setting, and hence $\alpha$ is surjective. The surjectivity of $\beta$ also follows from Lemma \ref{surjectiveness of multiplication map} in a similar fashion.
\vspace{0.2cm}

Case III: $|-K_S|$ has no fixed part and $K^2_S=0$.

\vspace{0.2cm}

In this case, $-K_S$ is nef, and hence base point free. We have an elliptic fibration $\pi: S\xrightarrow{|-K_S|} \prj{1}$ with connected fibres. Note that $\pi$ has (finitely many) singular fibers, because $\chi_{\text{top}}(S)=12\neq 0$. Also note that $(K_S+A)\cdot (-K_S)\ge 3$ by Proposition \ref{inequality1}.

If $H^1(\sshf{S}(B-\left\lfloor N\right\rfloor))=0$, we take $C=\left\lfloor N\right\rfloor$. Then as in Case II,  we are done. Therefore we assume $H^1(\sshf{S}(B-\left\lfloor N\right\rfloor))\neq 0$. This nonvanishing implies that
 \begin{equation*}
    P\cdot (-K_S)=P^2=0.
 \end{equation*}
We deduce from the Hodge index theorem that $P\equiv c(-K_S)$ for some rational number $c>0$. Let $E\in |-K_S|$ be a general fibre of $\pi$, we have that
\begin{equation*}
   \deg(\sshf{S}(B)|_E)= B\cdot (-K_S)=\frac{1}{c}(P+N)\cdot P=0.
\end{equation*}
We discuss by cases:
\vspace{0.2cm}

(i) $\sshf{S}(B)|_E\nsimeq \sshf{E}$, but topologically trivial.

\vspace{0.2cm}

In this subcase, we take $C=\left\lfloor N\right\rfloor+\left\lfloor c\right\rfloor E$. Then
\begin{equation*}
    B':=B-C\sim K_S+E'+\{c\}E+\{N\},
\end{equation*}
where $E'$ is another smooth member of $|-K_S|$, distinct from $E$.

The surjectivity of $\alpha$ will again follow from Lemma \ref{surjectiveness of multiplication map2}. We leave out the verification for intersection conditions in Lemma \ref{surjectiveness of multiplication map2}, but focus on showing that $H^1(\sshf{S}(B'))=0$.

First we assume that $\{N\}$ is simple normal crossing. Let $U\subset \prj{1}$ be the maximal open subset, over which $\pi$ is smooth. For each closed point $t\in U$, denote by $S_t$ the fibre over $t$. Since $\sshf{S}(B')|_E\iso \sshf{S}(B)|_E$ is nontrivial, thus
\begin{equation*}
    h^0(\sshf{S}(B')|_{S_t})=h^1(\sshf{S}(B')|_{S_t})=0.
\end{equation*}
Then by Grauert theorem \cite[III 12.9]{Hartshorne77}, $R^i\pi_*\sshf{S}(B')|_U=0$, for each $i\ge 0$. Thanks to the torsion freeness of the direct images $R^i\pi_*\sshf{S}(B')$, see \cite[Thm 1.1]{Fujino09}, we have that $R^i\pi_*\sshf{S}(B')=0$ for each $i$. By the Leray spectral sequence $E^{i, j}_2:=H^j(\prj{1}, R^i\pi_*\sshf{S}(B'))\Rightarrow H^{i+j}(S, \sshf{S}(B'))$, we obtain that $H^1(\sshf{S}(B'))=0$.

The general case for $\{N\}$ can be reduced to the simple normal crossing case by successive blowing ups as follows. For simplicity, assume that the non simple normal crossing locus for $\{N\}$ is a single point $x$, then after one blowing-up at $x$, we get a birational morphism $\mu: \tilde{S}\rightarrow S$ with the exceptional divisor $G$, such that $\mu^*\{N\}$ has simple normal crossing support. It holds that
\begin{equation*}
    \mu^*B'-aG\sim K_{\tilde{S}}+\overline{E'}+\{c\}\overline{E}+\overline{\{N\}}+sG,
\end{equation*}
where $\overline{*}$ denotes the strict transform of $*$, $a\ge 0$ is some integer, and $0\le s<1$, with the divisor $\overline{E'}+{c}\overline{E}+\overline{\{N\}}+sG$ being simple normal crossing. Thus applying the above argument to $\mu^*B'-aG$ on $\pi\circ\mu: \tilde{S}\rightarrow \prj{1}$ yields that $H^1(\tilde{S}, \sshf{\tilde{S}}(\mu^*B'-aG))=0$.

Now the sequence $\ses{\sshf{S}(B')\otimes \frak{m}^a_x}{\sshf{S}(B')}{\sshf{S}(B')\otimes\sshf{S}/{\frak{m}^a_x}}$ induces the surjection
\begin{equation*}
    H^1(\sshf{S}(B')\otimes \frak{m}^a_x)\rightarrow H^1(\sshf{S}(B'))\rightarrow 0.
\end{equation*}
We conclude that $H^1(\sshf{S}(B'))=0$ by observing that $ H^1(\sshf{S}(B')\otimes \frak{m}^a_x)\iso H^1(\tilde{S}, \sshf{\tilde{S}}(\mu^*B'-aG))$.

For the surjectivity of $\beta$, we apply Lemma \ref{surjectiveness of multiplication map} for $\left\lfloor N\right\rfloor$ ($\left\lfloor N\right\rfloor$ is rational by Proposition \ref{rational cycle}), and apply \cite[Prop. 5.4]{RS16} for $\left\lfloor c\right\rfloor E$, respectively.

\vspace{0.2cm}

(ii) $\sshf{S}(B)|_E\iso \sshf{E}$.

\vspace{0.2cm}

By a similar argument as in (i), we see $\pi_*(\sshf{S}(B))\iso \sshf{\prj{1}}(-m)$ is a line bundle for some $m\in \mathbb{Z}$. Therefore, by the projection formula, $\pi_*(\sshf{S}(B-mK_S))\iso \sshf{\prj{1}}$. The nonzero natural morphism
\begin{equation*}
    \sshf{S}\iso \pi^*\sshf{\prj{1}}\iso \pi^*\pi_*(\sshf{S}(B-mK_S))\rightarrow \sshf{S}(B-mK_S)
\end{equation*}
gives rise to a nonzero section $s\in \Gamma(S, \sshf{S}(B-mK_S))$, whose zero locus
\begin{equation*}
    (s)_0=\Gamma=\sum_i a_i\Gamma_i,
\end{equation*}
where $a_i\ge 0$ are integers and $\Gamma_i$ are irreducible. So we get the linear equivalence
\begin{equation*}
    B\sim mK_S+\Gamma.
\end{equation*}
Since $\Gamma_i\cdot (-K_S)=0$ and $h^0(\sshf{S}(\Gamma_i))=1$, we deduce that $\Gamma_i$ is contained in some reducible singular fibres. Moreover, by Lemma \ref{criterion for rational cycle}, $\Gamma_i\iso \prj{1}$.

Claim A: $\Gamma$ is a rational cycle.

\begin{proof}[Proof of Claim A]
From above, we know that
\begin{equation*}
    (B-mK_S)+K_S\sim K_S+\Gamma.
\end{equation*}
Suppose that $h^0(\sshf{S}(K_S+\Gamma))>0$. Then $h^0(\sshf{S}(B-mK_S))\ge h^0(\sshf{S}(-K_S))\ge 2$. This contradicts the fact that $\pi_*(\sshf{S}(B-mK_S))\iso \sshf{\prj{1}}$. Therefore $H^0(\sshf{S}(K_S+\Gamma))=0$, and the assertion follows from Lemma \ref{criterion for rational cycle}.
\end{proof}

Claim B: $h^0(\sshf{S}(l\Gamma))=1$ for each $l\ge 1$.

\begin{proof}[Proof of Claim B]
Put $\shf{L}=\sshf{S}(\Gamma)$. We first show that $h^1(S, \shf{L}^l)=1$ for all $l\ge 0$ by induction on $l$. Suppose $h^1(\shf{L}^l)=1$ for some $l\ge 1$. The exact sequence
\begin{equation*}
    \ses{\sshf{S}}{\sshf{S}(\Gamma)}{\sshf{\Gamma}(\Gamma)}
\end{equation*}
yields the exact sequence
\begin{equation*}
    H^1(S, \shf{L}^l)\rightarrow H^1(S, \shf{L}^{l+1})\rightarrow H^1(\shf{L}^{l+1}|_{\Gamma}).
\end{equation*}
Since the intersection $\shf{L}$ with each component $\Gamma_i$ is zero and $\Gamma$ is a rational cycle by Claim A, it follows that $H^1(\shf{L}^{l+1}|_{\Gamma})=0$ by a result of Harbourne \cite[Lemma 7]{Harbourne96}. Then by the induction hypothesis, we get that $H^1(S, \shf{L}^{l+1})=0$.

It is evident that $h^2(\shf{L}^l)=0$, so the Riemann--Roch theorem implies that
\begin{equation*}
    h^0(\shf{L}^l)=\chi(\shf{L}^l)=\frac{(l\Gamma)\cdot(l\Gamma-K_S)}{2}+1=1.
\end{equation*}
This completes the proof of Claim B.
\end{proof}

Next we shall show $m<0$.
By the assumption of the theorem, for $l\gg 0$,
\begin{equation*}
    K_S+lB\sim (lm+1)K_S+l\Gamma
\end{equation*}
is effective. Suppose $m\ge 0$, then
\begin{equation*}
    H^0(\sshf{S}(K_S+lB))\iso H^0(\sshf{S}(l\Gamma -(lm+1)E))=0,
\end{equation*}
because the linear system $|l\Gamma|$ does not move according to Claim B. This contradiction forces that $m<0$.

Finally we take $C=(-m)E+\Gamma$. Recall since $E\in|-K_S|$ is general, $E\cap \Gamma=\emptyset$. We have that
\begin{equation*}
    H^1(\sshf{S}(B-C))\iso H^1(\sshf{S})=0,
\end{equation*}
so the surjective of $\alpha$ follows from Lemma \ref{surjectiveness of multiplication map2}. As for that of $\beta$, we apply \cite[Prop. 5.4]{RS16} for $(-m)E$ and Lemma \ref{surjectiveness of multiplication map} for $\Gamma$, respectively.
\end{proof}

\bibliography{Ontheprojectivenormalityofcycliccoveringsoverarationalsurface}{}
\bibliographystyle{plain}

\end{document}